\documentclass[12pt,reqno]{amsart}
\usepackage{amssymb,amsthm}
\usepackage[T2A]{fontenc}
\usepackage[utf8]{inputenc}
\usepackage{bbm}
\usepackage{pdfsync}
\usepackage{ytableau}
\usepackage{graphicx}

\advance\hoffset by -0.5truein
\let\<\langle
\let\>\rangle

\def\Pois{\mathrm{Pois}}
\def\Lie{\mathrm{Lie}}

\def\GD {\mathop {\fam0 GD }\nolimits}
\def\SGD {\mathop {\fam0 SGD }\nolimits}

\def\Magma {\mathop {\fam0 Magma }\nolimits}

\def\Com{\mathrm{Com}}
\def\Nov{\mathrm{Nov}}

\def\Der{\mathrm{Der}}
\def\wt{\mathop {\fam 0 wt}\nolimits}

\newtheorem{definition}{Definition}
\newtheorem{theorem}{Theorem}
\newtheorem{lemma}{Lemma}

\newtheorem{example}{Example}

\title[Free special Gelfand---Dorfman algebra]{Free special Gelfand---Dorfman algebra}
\author[V. Gubarev, B. K. Sartayev]{V. Gubarev, B. K. Sartayev}

\thanks{The authors were supported by the grant of the President of the Russian Federation for young scientists (MK-1241.2021.1.1).}

\begin{document}

\sloppy

\begin{abstract}
A Gelfand---Dorfman algebra is called special if it can be embedded into a differential Poisson algebra.
We find a new basis of the free Novikov algebra.
With its help, we construct the monomial basis of the free special Gelfand---Dorfman algebra.
\end{abstract}


\maketitle

\section{Introduction}

A vector space~$A$ with a bilinear product~$\circ$ satisfying the identities
\begin{gather}
(x_1\circ x_2)\circ x_3-x_1\circ(x_2\circ x_3)=(x_2\circ x_1)\circ x_3-x_2\circ(x_1\circ x_3), \label{LeftSym} \\
(x_1\circ x_2)\circ x_3=(x_1\circ x_3)\circ x_2, \label{RightCom}
\end{gather}
is called a Novikov algebra.
Novikov algebras were introduced in the study of Hamiltonian operators concerning integrability of certain partial differential equations~\cite{GelDor79}.
Later, Novikov algebras appeared in the study of Poisson brackets of hydrodynamic type~\cite{BalNov85}.

It is well-known that given a commutative algebra $A$ with a derivation $d$,
the space $A$ under the product  
$x_1\circ x_2 = x_1d(x_2)$
is a Novikov algebra.
Moreover, all identities fulfilled in $(A,\circ)$ are consequences of~\eqref{LeftSym} and~\eqref{RightCom}. Applying the rooted trees, the monomial basis of the free Novikov algebra in terms of $\circ$ was constructed in~\cite{DzhLofwall}. In terms of Young diagram, the basis was constructed in~\cite{DzhIsmailov}.

An algebra~$A$ satisfying only the identity~\eqref{LeftSym} is called a left-symmetric algebra.
Left-symmetric algebras have been studying since 1960s, they have applications in affine geometry, ring theory, vertex algebras etc, see the survey~\cite{Burde06}. Left-symmetric algebras embeddable under the operation $x_1\circ x_2 = x_1d(x_2)$ into permutative algebras were studied in~\cite{KS2022}.

Note that every associative algebra is left-symmetric one, and every left-symmetric algebra under the commutator $[a,b] = a\circ b - b\circ a$ is a Lie algebra. 
For that reason, every Novikov algebra under the commutator is a Lie algebra satisfying an additional identity of degree~5 of the following form:
$$
\sum_{\sigma\in S_4}(-1)^{\sigma}[x_{\sigma(1)},[x_{\sigma(2)},[x_{\sigma(3)},[x_{\sigma(4)},x_5]]]]=0.
$$
To find all special identities for Novikov algebras considered under the commutator is still an open problem. 
It is equivalent to the same question formulated for a~commutative algebra~$C$ with a derivation~$d$
considered under the product
$$
[x_1,x_2] = x_1d(x_2) - x_2d(x_1),
$$
which is called Wronskian bracket.

Given a Poisson algebra~$(P,\cdot,\{,\})$ with a~derivation~$d$ due to both products, define on~$P$ new operations as follows,
$$
x_1\circ x_2 = x_1 d(x_2),\quad
[x_1,x_2] = \{x_1,x_2\}.
$$
Recall that the variety of Poisson algebras is defined by the identities,
\begin{gather}
x_1 \cdot x_2 = x_2\cdot x_1, \quad
(x_1 \cdot x_2)\cdot x_3 = x_1 \cdot (x_2\cdot x_3), \\
\{x_1,x_2\} = - \{x_2,x_1\}, \quad
\{\{x_1,x_2\},x_3\} + \{\{x_2,x_3\},x_1\} + \{\{x_3,x_1\},x_2\} = 0, \label{Lie-id} \\
\{x_1, x_2\cdot x_3\} = \{x_1,x_2\}\cdot x_3 + x_2\cdot\{x_1,x_3\}.
\end{gather}
The algebra~$P^{(d)}:=(P,\circ,[,])$ has a Novikov product~$\circ$, a Lie product~$[,]$, and moreover, 
the following identities hold,
\begin{gather}
x_2\circ[x_1,x_3]=[x_1,x_2\circ x_3]-[x_3,x_2\circ x_1]+[x_2,x_1]\circ x_3-[x_2,x_3]\circ x_1, \label{gd} \\
[x_1,(x_2\circ x_3)\circ x_4]=[x_1,x_2\circ x_3]\circ x_4+[x_1,x_2\circ x_4]\circ x_3-([x_1,x_2]\circ x_3)\circ x_4, \label{spec1} 
\end{gather}

\vspace{-0.7cm}

\begin{multline}\label{spec2}
[x_3\circ x_1,x_4\circ x_2]=[x_4\circ x_1,x_3\circ x_2]+[x_3,x_4\circ x_1]\circ x_2-[x_4,x_3\circ x_2]\circ x_1 \\ 
-[x_4,x_3\circ x_1]\circ x_2+[x_3,x_4\circ x_2]\circ x_1+2([x_4,x_3]\circ x_1)\circ x_2.
\end{multline}
There may exist identities of degree greater than~5 fulfilled in~$P$ which are independent from~\eqref{gd}--\eqref{spec2}.

An algebra $(G,\circ,[,])$ such that $(G,\circ)$ is Novikov, $(G,[,])$ is Lie, and the
identity~\eqref{gd} holds is called a Gelfand---Dorfman algebra ($\GD$-algebra)~\cite{WenHong, Xu2000}. 
$\GD$-algebras appeared in~\cite{GelDor79} as a source of Hamiltonian operators: 
with the help of the structure constants of a Gelfand–-Dorfman algebra one may construct a~differential operator.
In~\cite{Xu2000}, it was shown that $\GD$-algebras are closely related with Lie conformal algebras.

It is worth to note that the identities~\eqref{spec1} and~\eqref{spec2} are not fulfilled in the free GD-algebra, thus, these identities are called {\it special identities}. Moreover, the identities~\eqref{spec1} and~\eqref{spec2} are mutually independent \cite{KSO,KS}. 

A Gelfand---Dorfman algebra~$G$ is called {\it special} if there exists a Poisson algebra~$P$ with a derivation~$d$ such that $G$ injectively embeds into~$P^{(d)}$.
It is known that the class of special GD-algebras forms a variety~\cite{KS}, thus, we may consider
the free special Gelfand---Dorfman algebra~$\SGD\langle X\rangle$ generated by a~set $X$.
By $\Com\Der\langle X\rangle$ and $\Pois\Der\langle X\rangle$ we denote 
the free commutative and the free Poisson algebra with a derivation in the signature generated by~$X$, respectively.

In \cite{KS} it was proved that every 2-dimensional $\GD$-algebra is special. 
Another interesting result is that a $\GD$-algebra such that $[a,b] = a\circ b - b\circ a$ is special~\cite{KP2020}.

Therefore, a natural problem arises: 
To construct a monomial basis of the free $\SGD$-algebra in terms of $\circ$ and $[\;,\;]$. 
We solve this problem.
For the solution, we construct a new monomial basis of the free Novikov algebra.

For all mentioned varieties we have the following diagram:

\begin{picture}(30,80)
\put(195,57){$\hookrightarrow$}
\put(160,44){\normalsize\rotatebox[origin=c]{270}{$\hookrightarrow$}}
\put(195,32){$\hookrightarrow$}
\put(230,44){{\normalsize\rotatebox[origin=c]{270}{$\hookrightarrow$}}}
\put(143,57){$\Nov\langle X\rangle$}
\put(143,31){$\SGD\langle X\rangle$}
\put(213,57){$\Com\Der\langle X\rangle$}
\put(213,31){$\Pois\Der\langle X\rangle$}
\end{picture}
\vspace*{-\baselineskip}

In~\S2, we construct new basis of the free Novikov algebra (Theorem~2).
In~\S3, we define what is a canonical form  of monomials $\Pois\Der\langle X\rangle$ of weight~$-1$.
Finally, in~\S4, the linear basis of the free special $\GD$-algebra is constructed (Theorem~3).

For simplicity, we identify the element $d(x)$ with $x'$. 
In this paper, all algebras are defined over a field of characteristic~0.

\section{Basis of free Novikov algebra}\label{basisofNov}

Let $X = \{x_i \mid i\in I\}$, where $I$ is well ordered set.
The free commutative algebra $\Com\Der\langle X,d\rangle$ with a~derivation~$d$ in the signature has a standard linear basis consisting of monomials
$$
x_{i_1}^{(r_1)}\ldots x_{i_k}^{(r_k)}, \quad x_j\in X,\ i_1\leq \ldots\leq i_k,\ r_j\geq0.
$$
Here $x_j^{(0)} = x_j$, $x_j^{(n+1)} = \big(x_j^{(n)}\big)'$.
Thus, the elements $x^{(r)}$, where $x\in X$ and $r\in \mathbb{N}$, generate $\Com\Der\langle X,d\rangle$ as a commutative algebra.
For simplicity, we will denote $\Com\Der\langle X,d\rangle$ as $\Com\Der\langle X\rangle$
omitting the symbol~$d$.

\begin{definition}
Let $u$ be a monomial from the standard basis of $\Com\Der\langle X,d\rangle$.
Define the {\em weight} function $\wt(u)\in \mathbb Z$ by induction as follows,
\begin{gather*}
\wt(x)=-1,\quad x\in X; \\
\wt(d(u)) = \wt(u)+1; \quad \wt(uv)=\wt(u)+\wt(v).
\end{gather*}
\end{definition}

We may consider the space $\Com\Der\langle X\rangle$
under the product $u\circ v = ud(v)$, denote the obtained Novikov algebra as 
$\Com\Der\langle X\rangle^{(d)}$.
Let $\Com\Der\langle X\rangle_{-1}$ be a span of monomials from the standard basis of $\Com\Der\langle X\rangle$ of weight~$-1$. Note that $\Com\Der\langle X\rangle_{-1}$ is closed under the Novikov product $\circ$, so, it is a Novikov subalgebra of $\Com\Der\langle X\rangle^{(d)}$.

Let us recall the well-known results related to the free Novikov algebra.

\newpage

\begin{theorem} \label{thm:embedding}
a)~\cite{Umirbayev,DzhLofwall} 
We have $(\Com\Der\langle X\rangle_{-1},\circ)\cong \Nov\langle X\rangle$.

b)~\cite{BCZ2017,KS} Every Novikov algebra can be embedded into a free differential commutative algebra.
\end{theorem}

Let us define an order on the elements $x^{(r)}$ of $\Com\Der\langle X\rangle$ as follows:
$x_i^{(r_m)}>x_j^{(r_n)}$ if $r_m>r_n$ or $r_m=r_n$ and $i>j$.

We define a normal form of monomials $\Com\Der\langle X\rangle$ of weight $-1$ as follows,
\begin{equation}\label{goodwordnov}
x_{i_1}x_{i_2}\ldots x_{i_{l}}x_{j_n}^{(r_n)}\ldots x_{j_2}^{(r_2)}x_{j_1}^{(r_1)}x_{k_m}'\ldots x_{k_2}'x_{k_1}',
\end{equation}
where 
\begin{gather*}
n\geq1,\quad r_1,\ldots, r_n\geq2, \quad 
l=r_1+r_2+\ldots+r_n-n,\\ 
x_{i_1}\leq \ldots\leq x_{i_{l}}, \quad 
x_{j_n}^{(r_n)}\geq\ldots\geq x_{j_1}^{(r_1)},\quad x_{k_m}\geq \ldots \geq x_{k_1}.
\end{gather*}

Denote by $N(X)$ the set of all normal forms~\eqref{goodwordnov} of monomials from the standard basis of $\Com\Der\langle X\rangle$ of weight $-1$.

For $a\in N(X)$, put 
$$
L(a) = (r_1,\ldots,r_n),\quad
M(a) = (i_1,\ldots,i_l),\quad
R(a) = (k_m,\ldots,k_1),
$$
and define $S(a) = (L(a),R(a),M(a))$.
Given $a,b\in N(X)$, we say that $a<b$ if and only if $S(a)<S(b)$, we compare all tuples involved lexicographically.

Denote by $\Magma\langle X\rangle$ the free magma algebra with binary operation~$\circ$ generated by~$X$. We define a~linear map 
$\varphi\colon \Com\Der\langle X\rangle_{-1}\to \Magma\langle X\rangle$. 
By linearity it is enough to define $\varphi$ on the set $N(X)$, we do it inductively as follows,
\begin{equation} \label{ComDerPhiStart}
\varphi\big(x_{i_1}x_{i_2}\ldots x_{i_{n-1}}x_{i_n}^{(n-1)}\big)
 = x_{i_{n-1}}\circ(x_{i_{n-2}}\circ\ldots\circ(x_{i_1}\circ x_{i_n})\ldots),
\end{equation}
\begin{multline*}
\varphi\big(x_{i_1}x_{i_2}\ldots x_{i_{l}}x_{j_n}^{(r_n)}\ldots x_{j_2}^{(r_2)}x_{j_1}^{(r_1)}x_{k_m}'\ldots x_{k_2}'x_{k_1}'\big) \\
 = \varphi\big(x_{i_1}x_{i_2}\ldots x_{i_{l-r_n}}B_1x_{j_{n-1}}^{(r_{n-1})}\ldots x_{j_2}^{(r_2)}x_{j_1}^{(r_1)} x_{k_m}'\ldots x_{k_2}'x_{k_1}'\big),
\end{multline*}
where $B_1=\varphi\big(x_{i_{l+1-r_n}}x_{i_{l+2-r_n}}\ldots x_{i_l}x_{j_n}^{(r_n)}\big)$
is a new letter, so on this step we extend the generating set $X$ to $X_1 = X\cup \{B_1\}$. 
Thus, we define $\wt(B_1)=-1$ and $x<B_1$ for all $x\in X$. 
On each step~$i$, we add a new letter~$B_i$ to the generating set, i.\,e. $X_i = X_{i-1}\cup \{B_i\}$, we define $\wt(B_i)=-1$ and $y<B_i$ for all $y\in X_{i-1}$.
Calculating by the given rule, finally, we get
\begin{multline}\label{Novgeneralform}
\varphi\big(x_{i_1}x_{i_2}\ldots x_{i_{l}}x_{j_n}^{(r_n)}\ldots x_{j_2}^{(r_2)}x_{j_1}^{(r_1)}x_{k_m}'\ldots x_{k_2}'x_{k_1}'\big)\\
 = (\ldots((B_{n-1}\circ(x_{i_{r_1-1}}\circ\ldots(x_{i_1}\circ x_{j_1})\ldots))\circ x_{k_m})\ldots)\circ x_{k_1},
\end{multline}
where all previous letters $B_1,\ldots,B_{n-2}$ are inside $B_{n-1}$. 

\begin{example}
Let $x,y,z,t,q\in X$, $y>x$, $t>q$, then
$$
\varphi(xyz^{(2)}t'q')
 = \varphi(B_1 t'q')
 = \varphi(B_2 q')
 = B_2\circ q
 = (B_1\circ t)\circ q
 = ((y\circ (x\circ z))\circ t)\circ q.
$$
\end{example}

Define a homomorphism 
$$
\tau\colon \Magma\langle X\rangle\to \Com\Der\langle X\rangle_{-1} 
$$
by the formula 
$\tau(x) = x$, $x\in X$; the last algebra is considered as a~Novikov one.
For example, if $x,y,z\in X$, then
$\tau(x\circ (y\circ z)) = x(yz')' = xy'z' + xyz^{(2)}$.

\begin{lemma} \label{partitionderivation}
Let $a\in N(X)$. Then $\tau(\varphi(a)) = a + \sum_j b_j$, where $b_j<a$ for all~$j$.
\end{lemma}

\begin{proof}
By the definition of~$\tau$, it is enough to prove the statement for 
$$
a=x_{i_1}x_{i_2}\ldots x_{i_{l}}x_{j_n}^{(r_n)}\ldots x_{j_2}^{(r_2)}x_{j_1}^{(r_1)}.
$$
By the Leibniz rule fulfilled for~$d$, we have
\begin{multline*}
\tau\big(\varphi\big(x_{i_1}x_{i_2}\ldots x_{i_{l}}x_{j_n}^{(r_n)}\ldots x_{j_2}^{(r_2)}x_{j_1}^{(r_1)}\big)\big)
 = \tau(B_{n-1}\circ(x_{i_{r_1-1}}\circ(\ldots\circ(x_{i_1}\circ x_{j_1}))\ldots)) \\
 = \tau(B_{n-1})(\tau(x_{i_{r_1-1}}\circ(\ldots\circ(x_{i_1}\circ x_{j_1})\ldots)))' \\
 = \tau(B_{n-1})x_{i_1}\ldots x_{i_{r_1-1}}x_{j_1}^{(r_1)}
 + \sum_{p<r_1} \tau(B_{n-1}) \ldots x_{j_1}^{(p)},
\end{multline*}
and all summands are less than 
$x_{i_1}\ldots x_{i_{r_1-1}}\tau(B_{n-1})x_{j_1}^{(r_1)}$ due to the above defined order on~$N(X)$. Analogously, we deal with $\tau(B_{n-1})$ and so on.
\end{proof}

Define $N_\varphi = \{\varphi(a) \mid a\in N(X)\}$.

\begin{theorem}\label{newbasenov}
The set $N_\varphi$ forms a~basis of the free Novikov algebra~$\Nov\langle X\rangle$.
\end{theorem}

\begin{proof}
By Theorem~\ref{thm:embedding}a, we identify $\Nov\langle X\rangle$ with the Novikov algebra $\Com\Der\langle X\rangle_{-1}$.
Let $L$ be a linear span of $N_\varphi$ in $\Magma\langle X\rangle$.
We want to show that $\tau$ is an isomorphism of~$L$ and $\Nov\langle X\rangle$ considered as vector spaces.

By Lemma~\ref{partitionderivation}, we have that $\tau(\varphi(a)) = a + \sum_j b_j$ with $b_j<a$ for every $a\in N(X)$. Thus, we derive that $\tau(N_\varphi)$ is linearly independent.

Suppose that $\tau(N_\varphi)$ is not complete, so, the set $M = \{a\in N(X)\mid a$ is not expressed through $\tau(N_\varphi)\}$ is not empty. 
Choose a minimal $a\in M$ due to the order~$<$, 
such element exists, since the set of tuples $S(a)$ is well-ordered.
We have $a - \tau([a]) = \sum_j b_j$. 
By the assumption, all $b_j$ are expressed via $\tau(N_\varphi)$, so, $a$ is expressed too, a~contradiction.

So, $\tau\colon L\to \Nov\langle X\rangle$ is an isomorphism of vector spaces.
Thus, we may define the product on~$L$ by the formula
$n\circ m = \tau^{-1}(\tau(n)\tau(m)')$, where $n,m\in N_\varphi$.
Since $\tau$~is also a~homomorphism between algebras $L$ and $\Nov\langle X\rangle$, we have proved the statement.
\end{proof}

Let $n$ be a positive integer. We consider Young diagrams corresponding to the partitions 
$$
\lambda_1+\ldots+\lambda_k=n,\quad \lambda_1>\lambda_2\geq\ldots\geq \lambda_k\geq 1.
$$
We fill the Young diagrams by elements of $X$:
$$
\ytableausetup{mathmode, boxsize=2.7 em}
\begin{ytableau}
x_{i_{{1}_1}} & x_{i_{{1}_{2}}} & \dots & x_{i_{{1}_{\lambda_1-1}}} & x_{t_{1}} \\
\vdots \\
x_{i_{{r}_1}} & \dots & x_{i_{{r}_{\lambda_r-1}}} & x_{t_{r}} \\
x_{t_{r+1}}\\
\vdots \\
x_{t_{r+p}}\\
\end{ytableau}
$$
Here   
\begin{gather*}
i_{{1}_{\lambda_1-1}}\geq\ldots\geq i_{{1}_1} \geq\ldots\geq i_{{r}_{\lambda_r-1}}\geq\ldots\geq i_{{r}_1},\;\; t_{r+1}\geq\ldots\geq t_{r+p}, \\
t_{1}\geq t_{2}\;\textrm{if}\;\lambda_1=\lambda_{2}+1,\; \textrm{and}   \;t_{s}\geq t_{s+1}\; \textrm{if} \; \lambda_s=\lambda_{s+1}\; \textrm{for} \; s=2,\ldots,r-1.
\end{gather*}

For the diagram with exactly one row 
$(x_{i_{{1}_1}},x_{i_{{1}_{2}}},\dots,x_{i_{{1}_{\lambda_1-1}}},x_{t_{1}})$
we attach a monomial of $\Nov\langle X\rangle$ as follows,
$$
 u_1:=x_{i_{{1}_{\lambda_1-1}}}\circ (\ldots  \circ(x_{i_{{1}_{2}}}\circ(x_{i_{{1}_1}}\circ x_{t_{1}}))\ldots).
$$

For the diagram with $m$~rows we attach a~monomial of $\Nov\langle X\rangle$ inductively,
$$
m\textrm{-th row}\longrightarrow u_{m}
 :=u_{m-1}\circ (x_{i_{{m}_{\lambda_m-1}}}\circ(\ldots \circ (x_{i_{{m}_2}}\circ(x_{i_{m_1}}\circ x_{t_{m}}))\ldots)).
$$


The set of the constructed Young diagrams with the corresponding monomials coincides with the set $N_\varphi$.

\section{Normal form of monomials of weight~$-1$ in $\Pois\Der\langle X\rangle$}

Let $Y$ be a well-ordered set with respect to an order
$<$, and let $Y^*$ be the set of all associative words
in the alphabet $Y$ (including the empty word denoting by 1).
Extend the order to $Y^*$ by induction on the word length as follows.
Put $u<1$ for every nonempty word $u$. Further,
$u < v$ for $u = y_i u'$, $v = y_jv'$, $y_i, y_j\in Y$
if either $y_i < y_j$ or $y_i = y_j$, $u'< v'$.
In particular, the beginning of every word is greater than the whole word.

{\it Definition 2}.
A word $w\in Y^*$ is called an associative Lyndon---Shirshov word if
for arbitrary nonempty $u$ and $v$ such that $w=uv$,
we have $w>vu$.

For example, a word $aabac$ is an associative Lyndon---Shirshov word when $a>b>c$.

Consider the set $Y^+$ of all nonassociative words in $Y$,
here we exclude the empty word from consideration.

{\it Definition 3}.
A nonassociative word $[u]\in Y^+$ is called a nonassociative
Lyndon---Shirshov word (an LS-word, for short) provided that

(LS1) the associative word $u$ obtained from $[u]$
by eliminating all parentheses is an associative Lyndon---Shirshov word;

(LS2) if $[u] = [[u_1],[u_2]]$, then
$[u_1]$ and $[u_2]$ are LS-words, and $u_1> u_2$;

(LS3) if $[u_1] = [[u_{11}],[u_{12}]]$, then $u_2\geq u_{12}$.

These words appeared independently for the algebras and groups~\cite{Lyndon1958,Shirshov1958}.
In~\cite{Shirshov1958}, it was proved that the set of all
LS-words in the alphabet $Y$ is a~linear basis for a~free Lie
algebra generated by $Y$. 
Moreover, each associative Lyndon---Shirshov word~$w$ possesses the unique arrangement of parentheses which gives an LS-word $[w]$.

We consider the free Poisson algebra $\Pois\langle X\rangle$ generated by~$X$.
Here we denote the operations by~$x\cdot y$ and $\{x,y\}$.
Since $\Pois\langle X\rangle = \Com(\Lie \langle X\rangle)$, the set of commutative words
\begin{equation} \label{PoisFreeBasisForm}
A_1A_2\ldots A_n,\quad A_1\leq \ldots\leq A_n,
\end{equation}
where $A_i$ are Lyndon---Shirshov words in $\Lie\langle X\rangle$ forms a~standard basis of $\Pois\langle X\rangle$. 
By~\cite{KSO}, for the free Poisson algebra generated by a~set~$X$ with a derivation~$d$, we have the equality
$\Pois\Der\langle X\rangle = \Pois\langle X_\infty\rangle$, where
$X_\infty = \big\{x_i^{(n)}\mid i\in I,\,n\in\mathbb{N}\big\}$.

Define an order on $X_\infty$ as follows:
$x_i^{(m)}>x_j^{(n)}$ if $m>n$ or $m=n$, $i>j$.

We define an order on Lyndon---Shirshov words forming the basis in $\Lie\langle X_\infty\rangle$ as follows. At first, we compare two Lie words $A_1$ and $A_2$ by degree, i.\,e., $A_1>A_2$ if $\deg A_1>\deg A_2$. If $\deg A_1 = \deg A_2$, then we compare corresponding associative Lyndon---Shirshov words as it was defined above. 

Also, we define $A>x_{k}^{(m)}>B$, where $A$ and $B$ are LS-words on $X_\infty$ of degree at least two and $A$ but not $B$ involves~$d$ in its notation. 

Recall the definition of the weight function~\cite[Definition\,2]{KSO} on basic monomials~\eqref{PoisFreeBasisForm} with Lie words taken from $H(X,d)$ of~$\Pois\Der\langle X\rangle$,
\begin{gather*}
\wt(x)=-1,\quad x\in X; \\
\wt(d(u)) = \wt(u)+1; \
\wt(\{u,v\})=\wt(u)+\wt(v)+1; \
\wt(uv)=\wt(u)+\wt(v).
\end{gather*}

Due to~\cite{KSO}, we have
$(\Pois\Der\langle X\rangle_{-1},\circ,[,])\cong \SGD\langle X\rangle$, and the linear map 
$\xi\colon \Pois\Der\langle X\rangle_{-1}\to \SGD\langle X\rangle$ defined by the formulas $\xi(a\circ b)=ab'$, $\xi([a,b])=\{a,b\}$ provides the isomorphism.

Let us define a canonical form of monomials $\Pois\Der\langle X\rangle$ of weight $-1$ as follows:
\begin{equation}\label{poisgoodword}
x_{i_1}\ldots x_{i_k}B_1\ldots B_m A_n\ldots A_1 x_{j_l}^{(r_l)}\ldots x_{j_1}^{(r_1)},
\end{equation}
where $A_i,B_j$ are Lie-words of degree at least~2 and $A_i$ but not $B_j$ involves~$d$ in its notation, moreover,
\begin{gather*}
A_1\leq A_2\leq\ldots \leq A_n, \quad 
B_m\geq B_{m-1}\geq\ldots\geq B_1, \\
x_{i_k}\geq x_{i_{k-1}}\geq \ldots\geq x_{i_1}, \quad 
x_{j_l}^{(r_l)}\geq x_{j_{l-1}}^{(r_{l-1})}\geq \ldots \geq x_{j_1}^{(r_1)}.
\end{gather*}

Denote by $N(X)$ the set of all normal forms~\eqref{poisgoodword} of monomials from the standard basis of $\Pois\Der\langle X\rangle$ of weight $-1$.

Denote by $\Magma_2\langle X\rangle$ the free algebra with two binary (magma) operations~$\circ$ and~$[,]$ generated by~$X$. We define a~linear map 
$\psi\colon \Pois\Der\langle X\rangle_{-1}\to \Magma_2\langle X\rangle$ 
by induction.

At first, we consider a Lie LS-word corresponding to an associative LS-word 
$w = x_{i_1}^{(r_1)}\ldots x_{i_t}^{(r_t)}$ with $k = r_1+\ldots+r_t \geq 1$.
Let $\pi\in S_t$ be a permutation acting on the letters of $w\in X_\infty^*$ such that 
in the associative word $w^{\pi} = x_{j_1}^{(p_1)}\ldots x_{j_t}^{(p_t)}$
we have $x_{j_m}^{(p_m)}\geq x_{j_{m+1}}^{(p_{m+1})}$ for $m=1,\ldots,t-1$.

Given $c_k\geq \ldots\geq c_1$ such that $\wt(c_i)=-1$ 
(here by $c_i$ we mean either $x\in X$ or a Lie LS-word in $X$), we put
$$
\psi(c_{1}\ldots c_{k}[w])\\
 = [u],
$$
where the corresponding associative LS-word $u = v^\pi$ 
and the word $v = v_1 \ldots v_t$ is defined as follows,
$$
v_m = G(c_1,\ldots,c_k,w)_m 
 := \varphi\big(c_{k-p_1-\ldots-p_m+1}\ldots c_{k-p_1-\ldots-p_{m-1}} x_{j_m}^{(p_m)}\big), 
$$
the map~$\varphi$ is defined by~\eqref{ComDerPhiStart}. 
Here, we add new generators 
$G(c_1,\ldots,c_k,w)_m$ to $X$ for all $p_m\geq1$ to get the supset~$\widetilde{X}$. 
We compare new generators by the length in terms on the~$\circ$ operation 
and then after elimination of all signs of the~$\circ$ operation 
we compare them as associative words.
By this rule, all new letters are greater than elements from~$X$.
Also, we put $\wt(G(c_1,\ldots,c_k,w)_m) = -1$.

Let $u$ be a word of the form~\eqref{poisgoodword}. We define $\psi(u)$ as follows,
$$
\psi(u)
 = \psi\big(c_1\ldots c_m A_n \ldots A_1 x_{j_l}^{(r_l)}\ldots x_{j_1}^{(r_1)}\big) 
 = \psi\big(c_1\ldots c_{p_{n-1}} \tilde{A}_n A_{n-1}\ldots A_1 x_{j_l}^{(r_l)}\ldots x_{j_1}^{(r_1)}\big),
$$
where $\wt(c_{p_{n-1}+1}\ldots c_{m} A_n)=-1$, 
$\tilde{A}_n=\psi(c_{p_{n-1}+1}\ldots c_{m}\{A_n\})$ 
is a~letter of the extended alphabet~$\widetilde{X}$.

Thus, by $n$ such steps we exclude all Lie words of degree at least two which involve $d$ in their notation and afterwards apply the map~$\varphi$ from Sec.~2:
\begin{multline}\label{psiAction}
\psi(u)
 = \psi \big(c_1 \ldots c_{p_{n-2}}\tilde{A}_{n-1}A_{n-2}\ldots A_1 x_{j_l}^{(r_l)}\ldots x_{j_1}^{(r_1)}\big)
 = \ldots \\
 = \psi \big(c_1 \ldots c_{p_2}\tilde{A}_2A_1x_{j_l}^{(r_l)}\ldots x_{j_1}^{(r_1)}\big)
 = \varphi \big(c_1 \ldots c_{p_1}\tilde{A}_1x_{j_l}^{(r_l)}\ldots x_{j_1}^{(r_1)}\big). 
\end{multline}

\begin{example}
We have 
\[
\psi(x_3x_4x_5\{x_1',x_2''\})
 = [x_3\circ x_1,x_5\circ(x_4\circ x_2)],
\]
\begin{multline*}
\psi(x_5x_6x_7x_8\{x_3',x_4''\}\{x_1,x_2'\}x_9'')
 = \psi(x_5[x_6\circ x_3,x_8\circ(x_7\circ x_4)]\{x_1,x_2'\}x_9'') \\ 
 = \psi(x_5[x_1,[x_6\circ x_3,x_8\circ(x_7\circ x_4)]\circ x_2]x_9'')
 = [x_1,[x_6\circ x_3,x_8\circ(x_7\circ x_4)]\circ x_2]\circ(x_5\circ x_9).
\end{multline*}
\end{example}

\section{Basis of free $\SGD$-algebra}

Given $a,b\in N(X)$, we compare them as elements from $\Com\Der\langle X\rangle$ (see Sec.~2).

Define a homomorphism 
$$
\tau\colon \Magma_2\langle X\rangle\to \Pois\Der\langle X\rangle_{-1}
$$
by the formula 
$\tau(x) = x$, $x\in X$.

\begin{lemma}\label{partitionLieword}
Let $a\in N(X)$. Then $\tau(\psi(a)) = a + \sum_j b_j$, where $b_j<a$ for all~$j$.
\end{lemma}

\begin{proof}
Let $u = c_1\ldots c_m A_n \ldots A_1 x_{j_l}^{(r_l)}\ldots x_{j_1}^{(r_1)}\in N(X)$.
By the definition of~$\psi$ and $\varphi$, we have (see also~\eqref{psiAction})
\begin{multline*}
\psi(u)
 = \varphi \big(c_1 \ldots c_{q_1}\tilde{A}_1x_{j_l}^{(r_l)}\ldots x_{j_1}^{(r_1)}\big)
 = \varphi \big(c_1 \ldots c_{q_2}B_1 x_{j_{l-1}}^{(r_{l-1})}\ldots x_{j_1}^{(r_1)}\big)
 = \ldots \\
 = \varphi \big(c_1 \ldots c_{q_l}B_{l-1}x_{j_1}^{(r_1)}\big)
 = B_{l-1}\circ (c_{q_l}\circ(\ldots (c_1\circ x_{j_1})\ldots)).
\end{multline*}
Applying~$\tau$, we  get
\begin{equation} \label{Lemma2Induction}
\tau(\psi(u))
 = \tau(B_{l-1})(\tau( c_{q_l}\circ(\ldots (c_1\circ x_{j_1})\ldots) ))'.
\end{equation} 
By Lemma~1, 
$$
\tau( c_{q_l}\circ(\ldots (c_1\circ x_{j_1})\ldots) )
 = c_1 \ldots c_{q_l}x_{j_1}^{(r_1-1)} + \sum_{j} b_j, 
$$
where $b_j<c_1 \ldots c_{q_l}x_{j_1}^{(r_1-1)}$.
Writing down~$\tau(B_{l-1})$, we get the analogous expression for it as the right-hand side of~\eqref{Lemma2Induction}. 
By this remark and by the induction reasons, it is enough to figure out with~$\tau(\tilde{A}_1)$.
Let $A_1 = [v_1\ldots v_k]$, where $v_i = x_{j_i}^{(t_i)}$.
We have by Lemma~1,
$$
\tau(\tilde{A}_1)
 = \big[\tau(\tilde{A}_2)c_{l_1}\ldots c_{l_{t_1}} x_{j_1}^{(t_1)} + d_1,\ldots,c_{r_1}\ldots c_{r_{t_k}}x_{j_k}^{(t_k)} + d_k\big],
$$
where $d_i$ are less than the corresponding leading terms.
Extracting by the Leibniz rule $\tau(\tilde{A}_2)$ from the described Lie word as well as all $c_i$, we get the~$a$ and the sum of words less than~$a$ due to the defined order on $N(X)$.
\end{proof}

\begin{example}

If $u=x_3x_4\{x_1',x_2'\}$, then 
\begin{multline*}
\tau(\psi(x_3x_4\{x_1',x_2'\}))
 = \tau([x_3\circ x_1,x_4\circ x_2]) \\
 = \{x_3x_1',x_4x_2'\}
 = x_3x_4\{x_1',x_2'\} + \{x_3,x_4\}x_2'x_1' 
 + x_4\{x_3,x_2'\}x_1' + x_3\{x_1',x_4\}x_2'.
\end{multline*}
\end{example}

Define $N_\psi = \{\psi(a) \mid a\in N(X)\}$.

\begin{theorem}\label{basis[S]}
The set $N_\psi$ forms a~basis of the free $\SGD$-algebra~$\SGD\langle X\rangle$.
\end{theorem}

\begin{proof}
Analogously to the proof of Theorem~2, where we apply Lemma~2.
\end{proof}



Applying Theorem \ref{basis[S]}, we obtain the multiplication table in the free SGD-algebra. 
If $a,b\in N_{\psi}$ then we compute $a\circ b$ and $[a,b]$ as follows.
Firstly,
$$
\tau(a*b) = \sum_i \alpha_i c_i\in\Pois\Der\langle X\rangle,
$$
where $* = \circ$ or $* = [,]$, $\alpha_i \in F$ and $c_i\in N(X)$.
Secondly, by Theorem~\ref{basis[S]}, we have
$c_i = \tau\big(\sum_j \beta_{ij} d_{ij}\big)$
for some $\beta_{ij}\in F$ and $d_{ij}\in N_{\psi}$, which gives
$$
a*b = \sum_i\alpha_i\bigg(\sum_j \beta_{ij} d_{ij}\bigg).
$$

\begin{example}
We have that $[x_1,(x_2\circ x_3)\circ x_4]\notin N_{\psi}$ and 
\begin{multline*}
\tau([x_1,(x_2\circ x_3)\circ x_4])=\{x_1,x_2\}x_3'x_4'+\{x_1,x_3'\}x_2x_4'+\{x_1,x_4'\}x_2x_3' \\
 = \tau(([x_1,x_2]\circ x_3)\circ x_4
   + [x_1,x_2\circ x_3]\circ x_4
   - ([x_1,x_2]\circ x_3)\circ x_4 \\
   + [x_1,x_2\circ x_4]\circ x_3 - ([x_1,x_2]\circ x_3)\circ x_4),
\end{multline*}
which gives the identity (\ref{spec1}).

Also, $[x_2\circ x_3,x_1\circ x_4]\notin N_{\psi}$ and 
\begin{multline*}
\tau([x_4\circ x_1,x_3\circ x_2])
 = x_3x_4\{x_1',x_2\}+\{x_3,x_4\}x_2'x_1'+x_3\{x_4,x_2'\}x_1'+x_4\{x_1',x_3\}x_2' \\
 = \tau([x_3\circ x_1,x_4\circ x_2]-([x_3,x_4]\circ x_2)\circ x_1-[x_3,x_4\circ x_2]\circ x_1+([x_3,x_4]\circ x_2)\circ x_1 \\
 - [x_3\circ x_1,x_4]\circ x_2+([x_3,x_4]\circ x_2)\circ x_1+([x_3,x_4]\circ x_2)\circ x_1 \\
 + [x_4,x_3\circ x_2]\circ x_1 - ([x_4,x_3]\circ x_2)\circ x_1 + [x_4\circ x_1,x_3]\circ x_2 - ([x_4,x_3]\circ x_2)\circ x_1),
\end{multline*}
which gives the identity (\ref{spec2}).
\end{example}





\bigskip

\noindent Vsevolod Gubarev \\
Sobolev Institute of Mathematics \\
Acad. Koptyug ave. 4, 630090 Novosibirsk, Russia \\
Novosibirsk State University \\
Pirogova str. 2, 630090 Novosibirsk, Russia \\
e-mail: wsewolod89@gmail.com

\medskip

\noindent Bauyrzhan Kairbekovich Sartayev \\
Sobolev Institute of Mathematics \\
Suleyman Demirel University \\
Abylaikhan street, 1/1, 040900 Kaskelen, Kazakhstan \\
e-mail: baurjai@gmail.com

\end{document}